\documentclass[11pt,reqno]{amsart}

\usepackage{mathrsfs}
\usepackage{amsfonts}
\usepackage{amssymb}
\usepackage{amsthm}
\usepackage{amsmath}
\usepackage{latexsym}
\usepackage{tikz}
\usepackage{geometry}

\newtheorem{theorem}{Theorem}[section]
\newtheorem{proposition}[theorem]{Proposition}
\newtheorem{lemma}[theorem]{Lemma}

\newtheorem{remark}[theorem]{Remark}

\makeatletter \@addtoreset{equation}{section} \makeatother

\def\tilde{\widetilde}

\newcommand{\beq}{\begin{equation}}
\newcommand{\eeq}{\end{equation}}

\newcommand{\Rmnum}[1]{\expandafter\@slowromancap\romannumeral #1@}

\begin{document}

\title[Global classical solutions of 3D compressible viscoelastic system]
{Global classical solutions of 3D compressible viscoelastic system near equilibrium}

\author[Y. Zhu]{Yi Zhu}\address{Department of Mathematics, East China University of Science and Technology, Shanghai 200237, P.R.China}
\email{\tt zhuyim@ecust.edu.cn}

\date{}

\begin{abstract}
In this paper, we prove the global existence of general small solutions to compressible viscoelastic system.
We remove the ``initial state'' assumption ($\tilde \rho_0 \det F_0 =1$) and the ``div-curl'' structure assumption compared with previous works.
It then broadens the class of solutions to a great extent.
More precisely, the initial density state would not be constant necessarily, and no more structure is needed for global well-posedness theory. It's different from the elasticity system in which structure plays an important role.
Since we can not obtain dissipation information for  density  and deformation tensor directly, we introduce a new effective flux.
The core thought is regarding the wildest ``nonlinear term'' as ``linear term''.
Although the Sobolev norms of solution may increase now, we can still obtain the global existence for it.
\end{abstract}

\maketitle

\section{introduction}

\subsection{ Compressible viscoelastic system}

The three dimensional compressible viscoelastic system can be written as follows,
\begin{equation}\label{cvs}
\left\{
  \begin{array}{ll}
    \tilde{\rho}_t + \nabla\cdot ( \tilde{\rho} u)=0 , \\
    (\tilde{\rho} u)_t+\nabla\cdot (\tilde{\rho} u \otimes u)+\nabla P(\tilde{\rho}) =  \nabla \cdot(2\mu D(u))+\nabla(\lambda \nabla\cdot u)+ \nabla\cdot \Big( \frac{W_F(F)F^T}{\det F} \Big),\\
    F_t+u\cdot F=\nabla u F.
  \end{array}
\right.
\end{equation}
Here $\tilde{\rho}$ means the density and $u$ is the velocity field. $P(\tilde{\rho})$ stands for the scalar pressure which satisfies $P'(\tilde{\rho})>0$ for any $\tilde \rho$. $D(u)$ is the symmetric part of gradient  $\nabla u$, i.e., $\frac{1}{2}\big( \nabla u + (\nabla u)^T \big)$. $F$ means the deformation tensor and $W(F)$ is the elastic energy function  here. The coefficients $\mu$ and $\lambda$ are assumed to satisfy $ \mu > 0,  \frac{2}{3} \mu +\lambda > 0$.

In a special case of the Hookean linear elasticity, $W(F)=|F|^2$, the incompressible viscoelastic system takes the form \cite{chemin, hu2, linliu},

\begin{equation}\label{icvs}
\left\{
  \begin{array}{ll}
    F_t +u\cdot\nabla F =\nabla u F, \\
    u_t +u\cdot\nabla u +\nabla P = \mu \Delta u +\nabla\cdot (FF^T), \\
    \nabla \cdot u=0.
  \end{array}
\right.
\end{equation}

Recently,  both the compressible system \eqref{cvs} and the incompressible system \eqref{icvs} have been studied extensively.
To put our results in context, we shall first highlight some recent progress for these two systems.

For the incompressible system, we always consider the case deformation tensor is near a  nontrivial  equilibrium state and denote $U=F-I$ as the perturbation.
Here $I$ is the identity matrix.
Under some  ``div'' structure of initial data, i.e., $ \nabla\cdot U_0^T =0$, through  basic energy analysis, one can find the divergence part  $\nabla\cdot U$ is good  while the curl part $\nabla\times U$ behaves roughly which brings the main difficulty.
In 2D case, Lin, Liu and Zhang {\cite{linliu}} studied an auxiliary vector field and proved the global existence of small solutions with the ``div'' structure (we refer to \cite{llzhou2} for different approach).
While in 3D case, other more structure is needed to control the wildest part $\nabla\times U$.
In \cite{llzhou}, Lei, Liu and Zhou found a so-called ``curl'' structure
$$\nabla_k U_0^{ij}-\nabla_jU_0^{ik}=U_0^{lj}\nabla_lU_0^{ik}-U_0^{lk}\nabla_lU_0^{ij}, $$
which is  compatible with system \eqref{icvs}. Under both ``div'' structure and ``curl'' structure, they regard $\nabla\times U$ as higher order terms, and then prove  the results for  global small solutions in 2D and 3D.
Chen and Zhang \cite{cz} picked another ``curl-free'' structure  $\nabla \times (F_0^{-1}-I)=0 $ to solve the Cauchy problem.
Under the same structure assumption,  the initial-boundary value results were done by Lin and Zhang \cite{initial}.
Hu and Lin \cite{hulin} give the global weak solution with discontinuous initial data in 2D.
The result of critical $L^p$ framework was given by Zhang and Fang \cite{lp}. Feng, Zhu and Zi \cite{fzz} studied the blow up criterion. We also refer to \cite{fz, lzhou} for works about incompressible limit theory.
Recently, the author \cite{jfa} proved the global existence of small solutions to \eqref{icvs} without any physical structure in 3D.
We refer to \cite{old5, old6, old7, old1, old2, old3, old4} for results considering the related models.

Now, we turn to  the compressible case which is considered in this paper.
Using the identity of variation, indeed,
$$ \delta \det F=\det F \; \text{tr} (F^{-1}\delta F),$$
one get from the third equation of system \eqref{cvs},
$$ (\det F)_t + u\cdot\nabla (\det F) =\det F \nabla\cdot u .$$
Together with the first equality of system \eqref{cvs}, we can derive the  conservation type law,
$$  \big( \tilde{\rho} \det F \big)_t + u\cdot\nabla (\tilde{\rho}\det F)=0. $$
It then implies that,
\begin{equation}\nonumber
\tilde{\rho} \det F = \tilde{\rho}_0\det F_0,    \qquad (\tilde{\rho}, F)|_{t=0}=(\tilde{\rho}_0, F_0).
\end{equation}
Thus, under the following ``initial state'' assumption:
\begin{equation}\label{a2}
\tilde{\rho}_0 \det F_0=1,
\end{equation}
(notice here under Lagrangian coordinate, the ``initial state'' assumption is $\tilde{\rho}_0 =1$) the compressible viscoelastic system \eqref{cvs} with Hookean elastic material becomes the following new one,

\begin{equation}\label{cvs2}
\left\{
  \begin{array}{ll}
    \tilde{\rho}_t + \nabla\cdot ( \tilde{\rho} u)=0 , \\
    (\tilde{\rho} u)_t+\nabla\cdot (\tilde{\rho} u \otimes u)+\nabla P(\tilde{\rho}) = \nabla \cdot(2\mu D(u))+\nabla(\lambda \nabla\cdot u)+ \nabla\cdot ( \tilde{\rho} F F^T ),\\
    F_t+u\cdot F=\nabla u F.
  \end{array}
\right.
\end{equation}

 Similar to the incompressible case, if we assume that the deformation tensor is around an equilibrium state and denote $U=F-I$ as the perturbation, we can then analyze the linearized system of $(u, U)$. Through the basic energy method, one can find $\nabla \cdot U$ behaves well under the following compressible ``div'' structure,
\begin{equation}\label{div}
\nabla\cdot(\tilde{\rho}_0 F_0^T)=0.
\end{equation}
However $\nabla \times U$ is still out of control which brings the main difficulty. The following ``curl'' structure is then needed.
\begin{equation}\label{curl}
 F_0^{lk} \nabla_l F_0^{ij}-F_0^{lj}\nabla_l F_0^{ik}=0 .
\end{equation}
In \cite{zhang}, Qian and Zhang proved the global existence of small solutions to system \eqref{cvs} under the ``initial state'' assumption \eqref{a2} and the ``div-curl'' structure \eqref{div}, \eqref{curl}. They see the wildest part as higher order terms and then capture the weak dissipation. He and Xu obtained some result about bounded domain in \cite{xuli}. Hu and Wang \cite{hu} considered the multi-dimensional case.  We refer to
   \cite{com1, com2, hu3, yong, Gen, jjw, PX, Tan}  for more related  researches. By the line, all the theorems about the compressible viscoelastic system make full use of some structure assumptions.

The main goal of this paper is to establish the global existence of general small solutions to system \eqref{cvs} which mainly lies in the following two aspects.
\begin{itemize}
  \item Show that the initial density  state ($\tilde \rho_0 \det F_0$ or $\tilde \rho_0$ under Lagrangian coordinate) would not be constant necessarily. Indeed, we shall get rid of the ``initial state'' assumption \eqref{a2}. Moreover, we can deal with system \eqref{cvs} directly now.
  \item Show that the structure assumption would not be necessary in the global well-posedness theory of compressible viscoelastic system (it's different from the elasticity system in which structure plays an important role). Concretely, we will get rid of the ``div'' structure \eqref{div} which shall bring difficulty in deriving the dissipation of $\nabla\cdot U$ and get rid of the ``curl'' structure \eqref{curl} which shall bring difficulty in deriving the dissipation of $\nabla\times U$.
\end{itemize}

In a word, we'd like to broaden the class of solutions to compressible viscoelastic system \eqref{cvs} to a great extent.
 To achieve  this  goal, we want to show that under some proper transformation, the dissipation mechanism of  system \eqref{cvs} will be more clear even without any compatible condition.

Our main theorem can be stated as follows.

\begin{theorem}\label{thm1}
Consider the Cauchy problem of 3D compressible viscoelastic system \eqref{cvs} with Hookean linear elasticity, $W(F)={\frac{1}{2}}|F|^2$. And  for some suitable small constant $\epsilon$, the initial data $ (\tilde{\rho}, u, F)|_{t=0}=(\tilde{\rho}_0, u_0, F_0)$  satisfies,
\begin{equation}\label{initial}
\big\| |\nabla|^{-1} (\tilde{\rho}_0-1) \big\|_{H^3} + \big\| |\nabla|^{-1} u_0 \big\|_{H^3}+ \big\| |\nabla|^{-1} (F_0-I)\big\|_{H^3} \leq \epsilon,
\end{equation}
here $|\nabla |=(-\Delta)^{\frac{1}{2}}$. Then system \eqref{cvs} admits a unique global classical solution $(\tilde{\rho}, u, F)$ and holds that,

\begin{equation*}
\begin{split}
\sup_{0\leq t \leq \infty} \Big\{ \big\| |\nabla|^{s_0} & (\tilde{\rho}-1) \big\|_{H^{2-s_0}}^2 +
\big\| |\nabla|^{-1} u \big\|_{H^3}^2 + \big\| |\nabla|^{s_0} (F-I)\big\|_{H^{2-s_0}}^2 \Big\}\\
&  + \int_0^{\infty} (1+t)^2 \| \nabla^2 u\|_{H^1}^2 \;dt \leq \epsilon,
\end{split}
\end{equation*}
here $ 0< s_0< \frac{1}{2}$ is some positive constant.
\end{theorem}

\begin{remark}
In Theorem \ref{thm1}, we only assume that initial data has small disturbance near the equilibrium state. Without any structure assumption, we can still derive strong dissipation of velocity filed like before. However, for the deformation tensor $F$, we only obtain the uniform bound of high frequency part. An interesting and arduous question is, whether the low frequency part of $F$ admits uniform bound with such  general initial data. Moreover, whether we can expect decay information for the deformation tensor or not.
\end{remark}

\subsection{Transformation and analysis}
Without any structure assumption, it's difficult to derive the dissipation mechanism for the compressible viscoelastic system \eqref{cvs} directly. More extensively, without ``div'' structure, the basic energy balance law of $(\tilde \rho - 1, u, F-I)$ is destroyed even in the linearized system.
It seems impossible to obtain the damping mechanism of density $\tilde{\rho}$ and deformation tensor $F$. Even the uniform bound for $(\tilde{\rho}-1, F - I)$ is out of reach now.
Motivated by the thought of regarding the wildest ``nonlinear term'' as ``linear term'', we then turn to study some new quantity consists of the original ones.
To overcome the difficulty above, in this subsection we shall introduce a new quantity which is some suitable combination of effects from density and deformation tensor. We now give a brief overview of main ideas.

~\\~
{\bf{Step1. Reformulation and suitable dissipative system. }}

Without loss of generality, we set $\mu = 1, \lambda = 0$ and define $\rho=\tilde{\rho}-1$. We define the \textit{effective tensor} $G$ as follows,
$$ G =   \frac{FF^T}{\det F}-I - \big(P(\rho +1) - P(1)\big) I .$$
In the first step, we shall derive the evolution of a more compatible system. Notice the definition of pressure $P$ we can write,

\begin{equation}\label{P}
P_t + u\cdot \nabla P +P'(\rho+1) \; (\rho+1) \; \nabla \cdot u=0 .
\end{equation}
On the other hand, through the third equation of system \eqref{cvs} and careful calculation one can achieve,

\begin{equation}\label{FF}
\Big( \frac{FF^T}{\det F} \Big)_t + u\cdot\nabla \Big( \frac{FF^T}{\det F} \Big) = \nabla u \Big( \frac{FF^T}{\det F} \Big) + \frac{FF^T}{\det F} (\nabla u)^T -\nabla \cdot u \Big( \frac{FF^T}{\det F} \Big) .
\end{equation}
Combing \eqref{P} and \eqref{FF} together, now we can obtain the evolution of \textit{effective tensor} $G$,
\begin{equation}\label{G}
G_t + u \cdot \nabla G - q(\rho) \nabla \cdot u I + g(\nabla u, F) = \nabla u + (\nabla u)^T + P'(1) \nabla \cdot u I - \nabla \cdot  u I.
\end{equation}
Here the functions $q(\cdot)$ and $g(\cdot ,\cdot)$ take the form,
\begin{equation}\label{qg}
\begin{split}
q(\rho) =& \;  P' (\rho+1)\; (\rho+1) - P'(1),\\
g(\nabla u, F) = & - \nabla u \Big( \frac{FF^T}{\det F} - I \Big)  -
\Big(\frac{FF^T}{\det F} - I \Big) (\nabla u)^T \\
&  + \nabla \cdot u \Big( \frac{FF^T}{\det F} - I\Big) .
\end{split}
\end{equation}

Through the basic energy analysis, we  find that the system of $(u, G)$ is not good enough due to the extra term $P'(1) \nabla \cdot u I - \nabla \cdot  u I$, even we consider the following linearized system,
\begin{equation}\nonumber
\begin{cases}
u_t - \Delta u= \nabla \cdot G,\\
G_t = \nabla u + (\nabla u)^T + (P'(1) - 1) \nabla \cdot u I.
\end{cases}
\end{equation}
  It implies  that \textit{effective tensor} $G$ can not reveal the decay information of the compressible viscoelastic system immediately. Thus, to explore the dissipation mechanism of system, some other quantity is needed.

\begin{remark}
We should point out here, for the incompressible viscoelastic system \eqref{icvs}, this troublesome term $P'(1) \nabla \cdot u I - \nabla \cdot  u I$ is absent. It can be seen as the main difference  between the incompressible system and the compressible system in some sense. It's also the most difficult term as well.
\end{remark}

\begin{remark}
In the previous works, to deal with the wildest nonlinear term $\nabla \cdot \Big( \frac{W_F(F)F^T}{\det F}\Big)$, people usually regard it as some higher order terms. From this point of view, structure assumption is needed. In this paper, we just regard this ``nonlinear term'' as a ``linear term'' and obtain some new system.
\end{remark}

However, we find that the divergence part of $G$ seems more compatible with $u$. Now we  introduce the following \textit{effective flux} $\mathcal{G}$ as the new quantity,

\begin{equation}\label{GG}
\mathcal{G}=  \nabla \cdot G = \nabla \cdot \frac{FF^T}{\det F}-\nabla P .
\end{equation}
And the evolution of $\mathcal{G}$ can be derived from \eqref{G} easily,
\begin{equation}\label{Q}
\begin{split}
\mathcal{G}_t + \nabla \cdot (u \cdot \nabla G) - \nabla(q\nabla\cdot u) + \nabla \cdot g &= \Delta u + P'(1) \nabla \nabla \cdot u \\
&= \Delta \mathbb{P} u + (1 + P'(1) )\Delta \mathbb{P}^{\bot} u.
\end{split}
\end{equation}
Here, $\mathbb{P}=I+(-\Delta)^{-1}\nabla \nabla\cdot $ denotes the projection operator and $\mathbb{P}^{\bot}= -(-\Delta)^{-1}\nabla \nabla\cdot $.
Notice the unbalanced structure on the right hand side of \eqref{Q}, next we shall investigate the weighted velocity field $\mathbf{\Gamma} u = (\mathbb{P} + \sqrt a \; \mathbb{P}^\perp ) u$. Here $ a = 1 + P'(1)$ is a constant.
By now, the new quantities $(\mathbf{\Gamma} u, \mathcal{G})$ obey the following system,
    \begin{equation}\label{newSys}
    \left\{
      \begin{array}{ll}
        (\mathbf{\Gamma} u)_t + \mathbf{\Gamma} (u\cdot\nabla u) -\Delta \mathbf{\Gamma} u  -\mathbf{ \Gamma} \big\{ \big(\frac{1}{\tilde{\rho}}-1 \big)(\mathcal{G}+\Delta u) \big\} = \mathbf{\Gamma}  \mathcal{G} ,  \\
    \mathcal{G}_t + \nabla \cdot (u \cdot \nabla G)- \nabla (q \nabla \cdot u) + \nabla \cdot g(\nabla u, F) = \Delta \mathbf{\Gamma}^2 u.
      \end{array}
    \right.
    \end{equation}
The next analysis and disposition are based on this new system.

\quad\\~
{\bf{Step2. Energy frame in the new system of $(\mathbf{\Gamma} u, \mathcal{G})$.}}

According to the analysis for the following linearized system of $(\Gamma u, \mathcal{G})$,
    \begin{equation}\nonumber
    \left\{
      \begin{array}{ll}
        (\mathbf{\Gamma} u)_t  -\Delta \mathbf{\Gamma} u = \mathbf{\Gamma}  \mathcal{G} ,  \\
    \mathcal{G}_t  = \Delta \mathbf{\Gamma}^2 u,
      \end{array}
    \right.
    \end{equation}
we can give basic energy and decay energy for the new system \eqref{newSys}.  More precisely, for any positive time $t$, basic energy $\mathcal{E}(t)$ is defined as follows,
\begin{equation}\nonumber
\mathcal{E}(t)= \sup_{0 \leq t' \leq t} \Big\{  \big\| |\nabla|^{-1} \mathbf{\Gamma} u \big\|_{H^{3}}^2 + \big\|  |\nabla|^{-2} \mathcal{G} \big\|_{H^{3}}^2    \Big\}
+ \int_{0}^{t}  \Big( \|\mathbf{\Gamma} u\|_{H^{3}}^2 + \| |\nabla|^{-1} \mathcal{G}\|_{H^{2}}^2 \Big) \;dt' . \\
\end{equation}
Here we give the slightly dissipative energy $\mathcal{V}(t)$ like,
\begin{equation}\nonumber
 \begin{split}
 \mathcal{V}(t)=&  \sup_{0 \leq t' \leq t} (1+t')\Big\{  {\big\|\mathbf{\Gamma} u \big\|_{H^{2}}^2} + \big\||\nabla|^{-1}  \mathcal{G} \big\|_{H^{2}}^2    \Big\}  \\
&+ \int_{0}^{t} (1+t') \Big( \| \nabla\mathbf{\Gamma} u\|_{H^{2}}^2 + \|\mathcal{G}\|_{H^{1}}^2 \Big) \;dt' . \\
\end{split}
\end{equation}
And then, we define the strongly dissipative energy $\mathcal{W}(t)$ which reveals the damping  mechanism of system.
\begin{equation}\label{strongE}
 \begin{split}
  \mathcal{W}(t)=&  \sup_{0 \leq t' \leq t} (1+t')^2 \Big\{ \big\|  \nabla \mathbf{\Gamma} u   \big\|_{H^{1}}^2 + \| \mathcal{G} \|_{H^{1}}^2    \Big\}  \\
&+ \int_{0}^{t} (1+t') ^2 \Big( \| \nabla^2 \mathbf{\Gamma} u\|_{H^{1}}^2 + \|  \nabla \mathcal{G}\|_{L^2}^2 \Big) \;dt'.
\end{split}
\end{equation}
By the orthogonal property of  $\mathbb{P}$ and $ \mathbb{P}^\perp$ in the sense of $L^2$ inner product, for any real index $s$, it holds that,
$$\|\mathbf{\Gamma} u\|_{H^s}  \approx  \|u\|_{H^s}.$$

Naturally, the next step is to derive proper \textit{a priori} estimate for these energies. Notice that, the new system \eqref{newSys} is still coupled with the original bad quantities $(\rho, F - I)$.  The above energy frame is not self-closed. Luckily, thanks to the above reformulation,  to obtain the \textit{a priori} estimate of $(\mathbf{\Gamma} u, \mathcal{G})$,   we only need  the uniform bound of high frequency part for  $(\rho, F - I)$.

~\\~
{\bf{Step3. Back to the estimate of  original quantities $(\rho, F - I)$. }}

Recall that, without the ``div-curl'' condition \eqref{div} and \eqref{curl},  we shall only expect the uniform bound  of $(\rho, F - I)$.
By the strong dissipation of velocity    involved in   system \eqref{newSys},  we can expect the uniform bound for $(\rho, F - I)$ in the following form,
\begin{equation}\label{E3}
\mathcal{A}(t) =  \sup_{0 \leq t' \leq t} \Big\{  \big\||\nabla|^{s_0} \rho\big\|_{H^{2-s_0}}^2 + \big\|  |\nabla|^{s_0}(F - I) \big\|_{H^{2 - s_0}}^2    \Big\}.
\end{equation}
Here, $0 < s_0 < \frac{1}{2}$ is some positive constant. We call $\mathcal{A}(t)$ the assistant energy and it plays an important role in achieving the dissipation for both systems.  Compare \eqref{E3} with the setting of initial data in Theorem \ref{thm1} (i.e. \eqref{initial}), we will find that \eqref{E3} only shows the uniform bound of high frequency part for solution $(\rho, F - I)$. The main difficulty in getting the uniform bound of low frequency  part is the lack of basic balance law for $(u, \rho, F - I)$.
 It will also lead to the core estimate of this paper.

\begin{remark}
Notice here, since there is no more structure assumption, we can not obtain any decay information for $\rho$ and $F$.
Indeed, the Sobolev norms of density and deformation tensor may increase now.
\end{remark}

\section{Preliminaries}

In this section, we will first introduce some notations which are used all through the work.
Then, we give some useful propositions to help simplify the process of energy estimate.

\subsection{Notations}

Now, we introduce some notations. Since we consider the three dimensional space case, the derivative symbol $\nabla =  (\partial_1,\partial_2,\partial_3)$. The mixed time-space Sobolev norm $\| f\|_{L_t^p(X)}$ means,

\begin{equation}\nonumber
  \| f\|_{L_t^p(X)} \triangleq \big\|   \| f(t',x) \|_{X(\mathbb{R}^3)} \big\|_{L^p([0,t])} ,
\end{equation}
here $X(\mathbb{R}^3)$ means some space norm.
The intersect space norm is defined as,

\begin{equation}\nonumber
  \| f\|_{X\cap Y} \triangleq \| f \|_X + \|f \|_Y .
\end{equation}
The real number Sobolev norms $ \|\cdot \|_{H^s}$ and $ \|\cdot \|_{\dot{H}^s}$ are defined in the following,

\begin{equation}\nonumber
  \| f \|_{H^s}^2 \triangleq \int_{\mathbb{R}^3} (1+|\xi|^2)|\hat{f}(\xi)|^2 \; d\xi, \quad \| f \|^2_{\dot{H}^s}\triangleq \int_{\mathbb{R}^3} |\xi|^2|\hat{f}(\xi)|^2 \; d\xi.
\end{equation}
Here $s$ is a real number and $\hat{f}(\xi)$ denotes the Fourier transform function of $f(x)$ in $\mathbb{R}^3$.

Throughout this paper, we use $A \lesssim B$ to denote $A \leq \tilde{C} B$ for some absolute positive
constant $\tilde{C}$, whose meaning may change from line to line.

\subsection{Some useful propositions}

\begin{proposition}\label{prop0}
Assume that $\mathcal{E}(t) \ll 1$ holds on time interval $[0,T]$, then for any time $t \in [0,T]$ we have the following interpolation inequality,

\begin{equation}\nonumber
\mathcal{V}(t) \lesssim \mathcal{E}(t)^{\frac{1}{2}} \mathcal{W}(t)^{\frac{1}{2}}.
\end{equation}
\end{proposition}

\begin{proof}
This interpolation type inequality directly comes from  definition of Sobolev norms together with the standard Gagliardo-Nirenberg interpolation inequality. We omit the detailed process here for convenience.
\end{proof}
With the help of Proposition \ref{prop0}, we only need to give the estimate for $\mathcal{E}(t)$ and $\mathcal{W}(t)$ in the next section.

\begin{proposition}\label{prop1}
For any smooth function $f(\cdot)$ defined around zero with $f(0) = 0$, if we assume $ v \in H^2(\mathbb{R}^3)$ and $\|v\|_{H^2} < 1$, then there holds that,

\begin{equation}\label{prop22}
\begin{split}
\|f(v)\|_{L^{p_0}} \lesssim& \; \|v\|_{L^{p_0}}, \quad \forall \; 1 \leq p_0 \leq \infty,\\
\|\nabla^k f(v)\|_{L^2} \lesssim& \; \|\nabla^k v\|_{L^2}, \quad k=1,2.
\end{split}
\end{equation}
\end{proposition}
\begin{proof}
{First due to $\|v\|_{L^\infty} \lesssim \|v\|_{H^2}$ and boundedness theorem, then $f(v(x)), f'(v(x))$ and $f''(v(x))$ are uniformly bounded for $x \in \mathbb{R}^3$.} Notice the condition $f(0) = 0$ and  mean value theorem, we have $f(v) = f(0) + f'(r(v)) v = f'(r(v)) v$
 for some smooth function {$|r(v)| \lesssim |v|_{L^\infty}$}. Then, it's easy to get,

\begin{equation}\nonumber
\|f(v)\|_{L^{p_0}} \lesssim \; \|v\|_{L^{p_0}}, \quad \forall \;  1 \leq p_0 \leq \infty.
\end{equation}

For the second inequality in \eqref{prop22}, we only give the proof in the case $ k = 2$. The other cases can be handled in the similar way.
Using  H\"{o}lder inequality and Sobolev imbedding theorem, we directly have,
\begin{equation}\nonumber
\begin{split}
\|\nabla^2 f(v)\|_{L^2} =& \; \|\nabla (f'(v) \nabla v)\|_{L^2} \\
\lesssim& \; \| f''(v) |\nabla v|^2 \|_{L^2} +\| f'(v) \nabla^2 v\|_{L^2}\\
 \lesssim& \; \|f''(v)\|_{L^\infty} \|\nabla v\|_{L^3} \|\nabla^2 v\|_{L^2} + \|f'(v)\|_{L^\infty} \|\nabla^2 v\|_{L^2} \\
 \lesssim&  \; \|\nabla^2 v\|_{L^2}.
\end{split}
\end{equation}
Which completes the proof of this proposition.
\end{proof}

%
%

\begin{proposition}\label{prop2}
For any positive time $T$, we can derive the following time integral estimate of $u$ and $\mathcal{G}$.

\begin{equation}\nonumber
 \begin{split}
    \int_0^T \|u(t, \cdot)\|_{L^{p_0}} + \| |\nabla|^{-1}\mathcal{G}(t, \cdot)\|_{L^{p_0}} \;dt
     \lesssim & \; \mathcal{V}(T)^{\frac{1}{2}}+ \mathcal{W}(T)^{\frac{1}{2}}, \quad 6 < p_0 \leq \infty,\\
     {\int_0^T \|\nabla u(t, \cdot)\|_{\dot H^{p_1}} + \| |\nabla|^{-1} \mathcal{G}(t, \cdot)\|_{\dot H^{p_1}} \;dt }
     \lesssim & \; \mathcal{V}(T)^{\frac{1}{2}}+ \mathcal{W}(T)^{\frac{1}{2}}, \quad 1 < p_1 \leq 2,\\
    \int_0^T \|\nabla u(t, \cdot)\|_{L^{p_2}} + \|\mathcal{G}(t, \cdot)\|_{L^{p_2}} \;dt \lesssim& \; \mathcal{V}(T)^{\frac{1}{2}}+ \mathcal{W}(T)^{\frac{1}{2}}, \quad 2 < p_2 \leq \infty,\\
     {\int_0^T \|\nabla^2 u(t, \cdot)\|_{L^{p_3}} \;dt \lesssim} & \; {\mathcal{W}(T)^{\frac{1}{2}}, \quad 2 \leq  p_3 \leq 6.}
 \end{split}
\end{equation}

\end{proposition}

\begin{proof}
We now give the proof step by step. First, using Gagliardo-Nirenberg interpolation inequality, we know that for $p_0>6$ there holds,

\begin{equation}\nonumber
  \|u(t, \cdot)\|_{L^{p_0}} \lesssim \|\nabla u\|_{L^2}^{\frac{3}{p_0} + \frac{1}{2}} \|\nabla^2 u\|_{L^2}^{\frac{1}{2}-\frac{3}{p_0}}.
\end{equation}
Hence, using H\"{o}lder inequality we have,
\begin{equation}\nonumber
\begin{split}
\int_0^T \|u(t, \cdot)\|_{L^{p_0}}\; dt
\lesssim & \int_0^T \|\nabla u\|_{L^2}^{\frac{3}{p_0} + \frac{1}{2}} \|\nabla^2 u\|_{L^2}^{\frac{1}{2}-\frac{3}{p_0}} \; dt \\
 = & \int_0^T (1+t)^{\frac{3}{2p_0} + \frac{1}{4}}\|\nabla u\|_{L^2}^{\frac{3}{p_0} + \frac{1}{2}} (1+t)^{\frac{1}{2}-\frac{3}{p_0}}\|\nabla^2 u\|_{L^2}^{\frac{1}{2}-\frac{3}{p_0}} (1+t)^{{-(\frac{3}{4}- \frac{3}{2p_0})}} \; dt\\
\lesssim & \; \mathcal{V}(T)^{\frac{3}{2p_0} + \frac{1}{4}} \mathcal{W}(T)^{\frac{1}{4}-\frac{3}{2p_0}}.
\end{split}
\end{equation}
Applying Young's inequality, we then complete the proof for the first part. We omit the proof for the other terms since the estimates are in a similar way.
\end{proof}

The following product type estimates are often used in the next section.

\begin{proposition}\label{prop3}
For suitable smooth functions $h_1, h_2$ on $\mathbb{R}^3$, we have
\begin{equation}\label{prop3:eq1}
\begin{split}
\||\nabla|^{-1} (h_1 h_2)\|_{L^2(\mathbb{R}^3)} \lesssim &\; \| h_1 \|_{L^{\frac{3}{1+s_0}}} \|h_2\|_{L^{\frac{6}{3-2s_0}}}
\lesssim   \; \||\nabla|^{\frac{1}{2} - s_0}h_1\|_{L^2} \|\nabla^{s_0} h_2\|_{L^2}, \\
\|\nabla^k(h_1 h_2)\|_{L^2(\mathbb{R}^3)}  \lesssim & \; {\|h_1\|_{L^\infty} \|h_2\|_{\dot H^k} + \|h_1\|_{\dot H^k} \|h_2\|_{L^\infty}} , \quad k = 0, 1, 2.
\end{split}
\end{equation}
Here $0 < s_0 < \frac{1}{2}$ is some positive constant.
Moreover, if we set $h_1 = \mathcal{G}_i$, $\partial_j u_i$ or {$\Delta u_i$} ($i, j = 1, 2, 3$) as defined in \eqref{GG}
 and $h_2(t,x) \in [0, T] \times \mathbb{R}^3$ is some suitable smooth function, it holds that,

\begin{equation}\label{prop3:eq2}
\begin{split}
\int_0^T \||\nabla|^{-1} (h_1 h_2)\|_{L^2(\mathbb{R}^3)} \; dt  \lesssim & \int_0^T \|h_1\|_{L^\frac{3}{1 + s_0}} \|h_2\|_{L^\frac{6}{3-2s_0}} \; dt \\
\lesssim & \; \sup_{0  \leq t \leq T} \|\nabla^{s_0} h_2\|_{L^2} \big( \mathcal{V}(T)^{\frac{1}{2}}+ \mathcal{W}(T)^{\frac{1}{2}}\big),\\
\int_0^T \|{\nabla}(h_1 h_2)\|_{L^2(\mathbb{R}^3)} dt \lesssim &  \int_0^T {\|\nabla h_1\|_{L^2} \||\nabla|^\frac{1}{2}  h_2\|_{H^\frac{3}{2}}} \; dt \\
\lesssim & \; \sup_{0 \leq t \leq T}{\||\nabla|^\frac{1}{2}  h_2\|_{H^\frac{3}{2}}}\big( \mathcal{V}(T)^{\frac{1}{2}}+ \mathcal{W}(T)^{\frac{1}{2}}\big).
\end{split}
\end{equation}
\end{proposition}

\begin{proof}
By Sobolev imbedding theorem and H\"{o}lder inequality, there is,

\begin{equation}\nonumber
\begin{split}
  \||\nabla|^{-1} h_1 h_2\|_{L^2(\mathbb{R}^3)} \lesssim & \; \|h_1 h_2\|_{L^\frac{6}{5}} \\
  \lesssim & \; \|h_1\|_{L^\frac{3}{1 + s_0}} \|h_2\|_{L^\frac{6}{3-2s_0}}\\
\lesssim & \; \||\nabla|^{\frac{1}{2} - s_0}h_1\|_{L^2} \|\nabla^{s_0} h_2\|_{L^2}.
\end{split}
\end{equation}
It completes the proof for the first inequality in \eqref{prop3:eq1}. For the second inequality in \eqref{prop3:eq1},  {applying the Gagliardo-Nirenberg interpolation inequality, it has
\begin{equation}\nonumber
\begin{split}
\|\nabla^k(h_1h_2)\|_{L^2} \leq& \sum_{l = 0}^k C_k^l \|\nabla^l h_1 \nabla^{k-l} h_2\|_{L^2} \\
\lesssim & \sum_{l = 0}^k \|\nabla^l h_1\|_{L^{\frac{2k}{l}}} \|\nabla^{k-l} h_2\|_{L^{\frac{2k}{k-l}}} \\
\lesssim & \sum_{l = 0}^k \|h_1\|_{L^\infty}^\frac{k-l}{k} \|h_1\|_{\dot H^k}^{\frac{l}{k}} \|h_2\|_{L^\infty}^\frac{l}{k} \| h_2\|_{\dot H^k}^\frac{k-l}{k}\\
\lesssim & \|h_1\|_{L^\infty} \|h_2\|_{\dot H^k} + \|h_1\|_{\dot H^k} \|h_2\|_{L^\infty}.
\end{split}
\end{equation}
}
Now, let us focus on the  time integral estimates in \eqref{prop3:eq2}.
Using \eqref{prop3:eq1} and Proposition \ref{prop2}, noting the definition of $h_1$, $\mathcal{V}(t)$ and $\mathcal{W}(t)$ we can derive,

\begin{equation}\nonumber
\begin{split}
\int_0^T \||\nabla|^{-1} (h_1 h_2)\|_{L^2} \; dt \lesssim& \int_0^T \|h_1\|_{L^\frac{3}{1 + s_0}} \|h_2\|_{L^\frac{6}{3-2s_0}} \; dt \\
\lesssim &  \sup_{0 \leq t \leq T} \|h_2\|_{L^\frac{6}{3-2s_0}}\big( \mathcal{V}(T)^{\frac{1}{2}}+ \mathcal{W}(T)^{\frac{1}{2}}\big) \\
\lesssim &  \sup_{0  \leq t \leq T} \|\nabla^{s_0} h_2\|_{L^2}\big( \mathcal{V}(T)^{\frac{1}{2}}+ \mathcal{W}(T)^{\frac{1}{2}}\big).
\end{split}
\end{equation}
And for the second inequality in \eqref{prop3:eq2}, {due to H\"{o}lder's inequality and Sobolev's imbedding theorem, we have
\begin{equation}\nonumber
\begin{split}
\int_0^T \|\nabla(h_1 h_2)\|_{L^2} \; dt \lesssim& \int_0^T \|\nabla h_1\|_{L^2} \|h_2\|_{L^\infty} + \|h_1\|_{L^6}\|\nabla h_2\|_{L^3} \;dt \\
\lesssim &  \int_0^T \|\nabla h_1\|_{L^2} \||\nabla|^\frac{1}{2}  h_2\|_{H^\frac{3}{2}} \; dt \\
\lesssim & \sup_{0 \leq t \leq T}\||\nabla|^\frac{1}{2}  h_2\|_{H^\frac{3}{2}} \big( \mathcal{V}(T)^{\frac{1}{2}}+ \mathcal{W}(T)^{\frac{1}{2}}\big) .
\end{split}
\end{equation}
}
\end{proof}

\section{Energy estimate}
In this section, we shall derive \textit{a priori} estimates for the basic energy $\mathcal{E}(t)$, the strongly dissipative energy $\mathcal{W}(t)$ and the assistant energy $\mathcal{A}(t)$ respectively. Before the process, we give the total energy $\mathcal{E}_{total}(t)$ in the following,
$$ \mathcal{E}_{total}(t) \triangleq \mathcal{E}(t)+\mathcal{V}(t)+\mathcal{W}(t)+\mathcal{A}(t).$$
Since the estimate of assistant energy is clear and independent with other type energies, we will give the \textit{a priori} estimate for $\mathcal{A}(t)$ first.

In the next text, we always assume that $(\rho, u, F)$ is the smooth solution to system \eqref{cvs} on $[0, T]\times \mathbb{R}^3 $  and  $\mathcal{E}_{total}(T) \ll 1$.
\subsection{The estimate of assistant energy $\mathcal{A}(t)$}
The estimate of $\mathcal{A}(t)$ can be stated in the following lemma.

\begin{lemma}\label{lem}
Assume that energies are defined as in Section 1, we then have the following inequality,
\begin{equation}
\mathcal{A}(t)  \lesssim \mathcal{A}(0) + \mathcal{A}(t) \mathcal{E}_{total}(t)^{\frac{1}{2}} + \mathcal{A}(t)^{\frac{1}{2}}(\mathcal{V}(t)^\frac{1}{2} + \mathcal{W}(t)^\frac{1}{2}),
\end{equation}
holds for any positive time $t$.
\end{lemma}

\begin{proof}
Noting the first equation of system \eqref{cvs}, using H\"{o}lder inequality and Sobolev imbedding theorem, we directly get the following estimate,
\begin{equation}\label{rhot}
\begin{split}
& \frac{d}{dt} \big{\|} |\nabla|^{s_0} \rho \big{ \|}_{L^2}^{{^2}} \\
\lesssim & \; \big( \|u\cdot \nabla \rho\|_{\dot H^{s_0}} + \|\rho \nabla \cdot u\|_{\dot H^{s_0}} + \| u\|_{\dot H^{s_0 + 1}} \big) \| \rho\|_{\dot H^{s_0}}\\
\lesssim & \; \big(\|u\cdot \nabla \rho\|_{H^1} + \|\rho \nabla \cdot u\|_{H^1} + \| u\|_{\dot H^{s_0 + 1}} \big)\| \rho\|_{\dot H^{s_0}}\\
\lesssim & \; \big(\|u\|_{W^{1,\infty}} \|\nabla \rho\|_{H^1} + {\|\rho\|_{W^{1,3}} \|\nabla \cdot u\|_{L^6} + \|\rho\|_{L^\infty} \|\nabla \nabla \cdot u\|_{L^2}} + \| u\|_{\dot H^{s_0 + 1}}\big) \| \rho\|_{\dot H^{s_0}}.
\end{split}
\end{equation}
Notice here $0< s_0< \frac{1}{2}$ is a constant.
Hence, integrating \eqref{rhot} with time and applying Proposition \ref{prop2}, we can derive:
\begin{equation}\nonumber
\begin{split}
\|\rho\|_{\dot H^{s_0}}^{{^2}}  \lesssim & \; \|\rho_0\|_{\dot H^{s_0}}^{{^2}} + \mathcal{A}(t) \int_0^{t} \big(\|u\|_{W^{1,\infty} }+  {\|\nabla^2 u\|_{L^2}}\big) \; dt'\\
&  + \mathcal{A}(t)^{\frac{1}{2}} \int_0^{t} \|u\|_{\dot H^{s_0 + 1}} \; dt'\\
\lesssim & \;  \mathcal{A}(0) + \mathcal{A}(t) \mathcal{E}_{total}(t)^{\frac{1}{2}} + \mathcal{A}(t)^{\frac{1}{2}} (\mathcal{V}(t)^\frac{1}{2} + \mathcal{W}(t)^\frac{1}{2}).
\end{split}
\end{equation}
{
Similarly, we also have
\begin{equation}\nonumber
\begin{split}
\|F - I\|_{\dot H^{s_0}}^2  \lesssim &  \mathcal{A}(0) + \mathcal{A}(t) \mathcal{E}_{total}(t)^{\frac{1}{2}} + \mathcal{A}(t)^{\frac{1}{2}} (\mathcal{V}(t)^\frac{1}{2} + \mathcal{W}(t)^\frac{1}{2}).
\end{split}
\end{equation}
The estimates for higher-order derivatives in $\mathcal{A}(t)$ can be obtained as follows:
\begin{equation}\nonumber
\begin{split}
& \frac{1}{2}\frac{d}{dt} \big{\|} \nabla^2 \rho \big{ \|}_{L^2}^2 \\
= & \; -\int_{\mathbb{R}^3} \nabla^2 (u\cdot \nabla \rho) \nabla^2 \rho \; dx - \int_{\mathbb{R}^3} \nabla^2(\rho \nabla \cdot u) \nabla^2 \rho \; dx - \int_{\mathbb{R}^3} \nabla^2 \nabla \cdot u \nabla^2 \rho \; dx\\
\lesssim & \|\nabla u\|_{L^\infty} \|\nabla^2 \rho\|_{L^2}^2 + \|\nabla^2 u\|_{L^3} \|\nabla \rho\|_{L^6} \|\nabla^2 \rho\|_{L^2}\\
 & + \|\nabla^3 u\|_{L^2}\|\rho\|_{L^\infty}\|\nabla^2 \rho\|_{L^2} + \|\nabla^3 u\|_{L^2} \|\nabla^2 \rho\|_{L^2},
\end{split}
\end{equation}
and
\begin{equation}\nonumber
\begin{split}
& \frac{1}{2}\frac{d}{dt} \big{\|} \nabla^2 (F-I) \big{ \|}_{L^2}^2 \\
\lesssim & \|\nabla u\|_{L^\infty} \|\nabla^2 (F - I)\|_{L^2}^2 + \|\nabla^2 u\|_{L^3} \|\nabla (F-I)\|_{L^6} \|\nabla^2 (F-I)\|_{L^2}\\
 & + \|\nabla^3 u\|_{L^2}\|F-I\|_{L^\infty}\|\nabla^2 (F-I)\|_{L^2} + \|\nabla^3 u\|_{L^2} \|\nabla^2 (F-I)\|_{L^2}.
\end{split}
\end{equation}
Hence,
\begin{equation}\nonumber
\begin{split}
\|\rho\|_{\dot H^2}^2 + \|F - I\|_{\dot H^2}^2 \lesssim & \mathcal{A}(0) + \mathcal{A}(t)
  \int_0^t (\|\nabla u\|_{L^\infty} + \|\nabla^2 u\|_{L^3} + \|\nabla^3 u\|_{L^2} )\; d\tau\\
 &  + \mathcal{A}(t)^\frac{1}{2} \int_0^t \|\nabla^3 u\|_{L^2} \; d\tau \\
\lesssim & \mathcal{A}(0) + \mathcal{A}(t) \mathcal{E}_{total}(t)^{\frac{1}{2}} + \mathcal{A}(t)^{\frac{1}{2}}  \mathcal{W}(t)^\frac{1}{2}.
\end{split}
\end{equation}
Then, we complete the proof of this lemma.
}
\end{proof}

\subsection{The estimate of basic energy $\mathcal{E}(t)$}

Now, we turn to deal with the basic energy $\mathcal{E}(t)$.
Noting the second equation in system \eqref{cvs} and  the definition of $\mathcal{G}$ in Section 1, we can write (without loss of generality, set $\mu = 1, \lambda = 0$),
\begin{equation}\label{u}
u_t + u\cdot\nabla u - \Delta u - \Big(\frac{1}{\tilde{\rho}}-1 \Big)(\mathcal{G}+\Delta u) = \mathcal{G}.
\end{equation}
Applying operator $\mathbf{\Gamma}  = \mathbb{P} + \sqrt a \; \mathbb{P}^\perp $ on the evolution of $u$ in \eqref{u}, we now derive the following system of $(\mathbf{\Gamma} u, \mathcal{G})$,

    \begin{equation}\label{E0:eq1}
    \left\{
      \begin{array}{ll}
        (\mathbf{\Gamma} u)_t + \mathbf{\Gamma} (u\cdot\nabla u) -\Delta \mathbf{\Gamma} u  -\mathbf{ \Gamma} \big\{ \big(\frac{1}{\tilde{\rho}}-1 \big)(\mathcal{G}+\Delta u) \big\} = \mathbf{\Gamma}  \mathcal{G} ,  \\
    \mathcal{G}_t + \nabla \cdot (u \cdot \nabla G)- \nabla (q \nabla \cdot u) + \nabla \cdot g(\nabla u, F) = \Delta \mathbf{\Gamma}^2 u.
      \end{array}
    \right.
    \end{equation}
Here the functions $q(\cdot), g(\cdot , \cdot)$ are defined in \eqref{qg}.

The estimate of $\mathcal{E}(t)$ can be stated in the following lemma.

\begin{lemma}\label{lemE0}
Assume that energies are defined as in Section 1, we then have the following inequality,

\begin{equation}\nonumber
\mathcal{E}(t) \lesssim \mathcal{E}(0) + \mathcal{E}_{total}(t)^{\frac{3}{2}},
\end{equation}
holds for any positive time $t$.
\end{lemma}

\begin{proof}

The proof of Lemma \ref{lemE0} consists of two independent energy estimates. Firstly, we define partial basic energy $\tilde{\mathcal{E}}(t)$ as follows,

\begin{equation}\label{e01}
\tilde{\mathcal{E}}(t) \triangleq \sup_{0 \leq t' \leq t} \big( \||\nabla|^{-1} \mathbf{\Gamma} u\|_{H^3}^2 + \||\nabla|^{-2}\mathcal{G}\|_{H^3}^2 \big) + \int_0^t \| \mathbf{\Gamma} u\|_{H^3}^2 \; dt' .
\end{equation}
We shall first derive the estimate of $\tilde{\mathcal{E}}(t)$ .

\textbf{Step One:}

Applying derivatives $\nabla^k |\nabla|^{-1}$ ($k = 0, 1, 2, 3$) on the first equation of system \eqref{E0:eq1} and taking inner product with $\nabla^k |\nabla|^{-1} \mathbf{\Gamma} u$. Then applying derivatives
$\nabla^k |\nabla|^{-2}$ on the second equation of  \eqref{E0:eq1} and taking inner product with $\nabla^k |\nabla|^{-2}  \mathcal{G}$. Summing them up, there is,
\begin{equation}\label{E0:eq2}
\frac{1}{2}\frac{d}{dt} \Big(\| |\nabla|^{-1} \mathbf{\Gamma} u\|_{H^3}^2 + \||\nabla|^{-2} \mathcal{G}\|_{H^3}^2 \Big) + \|\mathbf{\Gamma} u\|_{H^3}^2 = \sum_{k=1}^5 \mathcal{E}_k,
\end{equation}
where,
\begin{equation}\nonumber
\begin{split}
\mathcal{E}_1 = & \sum_{k = 0}^{3} \Big\{ \int_{\mathbb{R}^3} \nabla^k |\nabla|^{-1} \mathbf{\Gamma} \mathcal{G} \nabla^k |\nabla|^{-1} \mathbf{\Gamma} u \; dx + \int_{\mathbb{R}^3} \nabla^k |\nabla|^{-2} \Delta \mathbf{\Gamma}^2 u  \nabla^k |\nabla|^{-2} \mathcal{G} \; dx\Big\},\\
\mathcal{E}_2 = & -\sum_{k = 0}^{3}\Big\{\int_{\mathbb{R}^3} \nabla^k |\nabla|^{-1} \mathbf{\Gamma} (u \cdot \nabla u) \nabla^k |\nabla|^{-1} \mathbf{ \Gamma} u\; dx  \\
& \qquad + \int_{\mathbb{R}^3} \nabla^k |\nabla|^{-2}  \nabla \cdot g(\nabla u, F) \nabla^k |\nabla|^{-2}  \mathcal{G}\; dx \Big\}, \\
\mathcal{E}_3 = &\sum_{k = 0}^{3} \int_{\mathbb{R}^3} \nabla^k |\nabla|^{-1} \mathbf{\Gamma} \Big\{ \big(\frac{1}{\tilde{\rho}}-1 \big)({\mathcal{G}+\Delta u}) \Big\} \nabla^k |\nabla|^{-1} \mathbf{\Gamma} u\; dx,\\
\mathcal{E}_4 = & \sum_{k = 0}^{3} \int_{\mathbb{R}^3} \nabla^k |\nabla|^{-2} \nabla (q \nabla \cdot u) \nabla^k |\nabla|^{-2} \mathcal{G} \; dx, \\
\mathcal{E}_5 = & - \sum_{k = 0}^{3} \int_{\mathbb{R}^3} \nabla^k |\nabla|^{-2} \nabla \cdot (u \cdot \nabla G) \nabla^k |\nabla|^{-2} \mathcal{G} \; dx.
\end{split}
\end{equation}
We then turn to give the bounds for the five terms on the right-hand side of \eqref{E0:eq2}.

The first term $\mathcal{E}_1$ just contains appropriate cancellation. Using integration by parts we can derive,
\begin{equation}\label{eq:j1}
\begin{split}
\mathcal{E}_1 = & \sum_{k = 0}^{3}\Big\{ \int_{\mathbb{R}^3} \nabla^k |\nabla|^{-1} \mathbf{\Gamma} \mathcal{G} \nabla^k |\nabla|^{-1} \mathbf{\Gamma} u \; dx + \int_{\mathbb{R}^3} \nabla^k |\nabla|^{-2} \Delta \mathbf{\Gamma} u  \nabla^k |\nabla|^{-2} \mathbf{\Gamma} \mathcal{G} \; dx\Big\}\\
= & \sum_{k = 0}^{3}\Big\{ \int_{\mathbb{R}^3} \nabla^k |\nabla|^{-1} \mathbf{\Gamma} \mathcal{G} \nabla^k |\nabla|^{-1} \mathbf{\Gamma} u \; dx - \int_{\mathbb{R}^3} \nabla^k |\nabla|^{-1} \mathbf{\Gamma} u  \nabla^k |\nabla|^{-1} \mathbf{\Gamma} \mathcal{G} \; dx\Big\}\\
= & \; 0.
\end{split}
\end{equation}

For the next term $\mathcal{E}_2$, using H\"{o}lder inequality and keeping $\|\mathbf{\Gamma} v\|_{L^2} \sim \|v\|_{L^2}$ in mind, it easily yields that,
\begin{equation}\nonumber
\mathcal{E}_2 \lesssim  \big\||\nabla|^{-1} (u\cdot \nabla u) \big\|_{H^3} \||\nabla|^{-1} u\|_{H^3} +
 \big\| |\nabla|^{-1} g(\nabla u, F) \big\|_{H^3}\||\nabla|^{-2}   \mathcal{G}\|_{H^3}.
\end{equation}
Applying  Proposition \ref{prop1} and Proposition \ref{prop3} to the inequality above, we then get the following estimate,

\begin{equation}\label{eq:j2}
\begin{split}
\int_0^t | \mathcal{E}_2(t') | \;dt'
 \lesssim & \; \big\| |\nabla|^{-1} u \big\|_{L_t^\infty (H^3)} \cdot \int_0^t \big{\|} |\nabla|^{-1} (u\cdot \nabla u) \big{\|}_{H^3} \;dt' \\
 &  + \big\| |\nabla|^{-2}   \mathcal{G} \big\|_{L^\infty_t(H^3)} \cdot \int_0^t \big\| |\nabla|^{-1} g(\nabla u, F) \big\|_{H^3} \; dt'\\
\lesssim &\| |\nabla|^{-1} u \big\|_{L_t^\infty (H^3)} {\big(\|u\|_{L^\frac{3}{1+s_0}} \|\nabla u\|_{L^\frac{6}{3-2s_0}} + \|u\|_{L^\infty} \|\nabla u\|_{H^2} + \|u\|_{H^2} \|\nabla u\|_{L^\infty} \big)} \\
& + \| |\nabla|^{-2}   \mathcal{G} \big\|_{L^\infty_t(H^3)}{ \big(\|F-I\|_{L^\frac{3}{1+s_0}} \|\nabla u\|_{L^\frac{6}{3-2s_0}} + \|F-I\|_{L^3}\|\nabla u\|_{L^6}} \\
& {+ \|F-I\|_{L^\infty} \|\nabla^2 u\|_{H^1} + \|\nabla F\|_{H^1} \|\nabla u\|_{L^\infty}\big)} \\
\lesssim & \; \big(\mathcal{V}(t)^{\frac{1}{2}} + \mathcal{W}(t)^{\frac{1}{2}} \big) { \Big [  \|u\|_{L_t^\infty(H^2)}\big\||\nabla|^{-1} u \big\|_{L_t^\infty (H^3)}} \\
& {+  \||\nabla|^{s_0}(F-I)\|_{L_t^\infty( H^{2 - s_0})}   \big\||\nabla|^{-2}   \mathcal{G} \big\|_{L^\infty_t(H^3)}\Big ]}\\
\lesssim & \; \big(\mathcal{V}(t)^{\frac{1}{2}} + \mathcal{W}(t)^{\frac{1}{2}} \big)   \big( \mathcal{E}(t)+ \mathcal{A}(t)^\frac{1}{2}\mathcal{E}(t)^\frac{1}{2} \big).
\end{split}
\end{equation}
Now, we turn to  $\mathcal{E}_3$. We denote $ \frac{1}{\tilde{\rho}}-1 = \frac{1}{\rho+1}-1 \triangleq f(\rho)$ which is a smooth function of $\rho$ and satisfies the condition $f(0)=0$. Then it's easy to get,

\begin{equation}\nonumber
\begin{split}
\mathcal{E}_3 \lesssim & \; \big\| |\nabla|^{-1}\big\{ f(\rho)({\mathcal{G}+\Delta u} ) \big\} \big\|_{L^2} \big\| |\nabla|^{-1} u \big\|_{L^2} \\
& \quad  + \big\|  f(\rho)({\mathcal{G}+\Delta u})  \big\|_{H^1} \|u\|_{H^3}.\\
\end{split}
\end{equation}
Using Proposition \ref{prop3} { and H\"{o}lder's inequality},
\begin{equation}\nonumber
\begin{split}
& \int_0^t |\mathcal{E}_3(t')| \; dt' \\
\lesssim & \; \big\| |\nabla|^{-1} u \big\|_{L^\infty_t (L^2)} \cdot \int_0^t \big\| |\nabla|^{-1}\big\{ f(\rho)({\mathcal{G}+\Delta u}) \big\} \big\|_{L^2}\; dt'  \\
&+ \big\|  f(\rho)({\mathcal{G}+\Delta u})  \big\|_{L^2_t(H^1)} \|u\|_{L^2_t(H^3)}\\
\lesssim & {\big\| |\nabla|^{-1} u \big\|_{L^\infty_t (L^2)} \|f(\rho)\|_{L_t^\infty (L^\frac{6}{3-2s_0}) }\|\mathcal{G}+\Delta u\|_{L_t^1(L^\frac{3}{1+s_0})}} \\
& {+ \Big(\|f(\rho)\|_{L^\infty_{t,x}} \|\mathcal{G}+\Delta u\|_{L_t^2(H^1)} + \|\nabla f(\rho)\|_{L^\infty_t(L^3)} \|\mathcal{G}+\Delta u\|_{L^2_t(L^6)}\Big) \|u\|_{L^2_t(H^3)}}\\
\lesssim & \;
\big(\mathcal{V}(t)^{\frac{1}{2}} + \mathcal{W}(t)^{\frac{1}{2}} \big) \mathcal{E}(t)^\frac{1}{2} { \|f(\rho) \|_{L^\infty_t(L^\frac{6}{3-2s_0})}} + \mathcal{E}(t) {\|f(\rho)\|_{L^\infty_t(L^\infty \cap \dot W^{1,3})}} . \\
\end{split}
\end{equation}
To give the estimate for  $f(\rho)$ we can use Proposition \ref{prop2} and directly get,
\begin{equation}\label{eq:j3}
\begin{split}
& \int_0^t |\mathcal{E}_3(t')| \; dt' \\
\lesssim & \;
\big(\mathcal{V}(t)^{\frac{1}{2}} + \mathcal{W}(t)^{\frac{1}{2}} \big) \mathcal{E}(t)^\frac{1}{2} \|\rho \|_{L^\infty_t(\dot H^{s_0} \cap \dot H^1)} + \mathcal{E}(t)  \|\rho\|_{L^\infty_t(\dot H^{s_0} \cap \dot H^2)}  \\
\lesssim & \; \big(\mathcal{V}(t)^{\frac{1}{2}} + \mathcal{W}(t)^{\frac{1}{2}} \big) \mathcal{E}(t)^\frac{1}{2} \mathcal{A}(t)^{\frac{1}{2}} + \mathcal{A}(t)^{\frac{1}{2}}\mathcal{E}(t).
\end{split}
\end{equation}

Similarly, for the next term $\mathcal{E}_4$, according to Propositions \ref{prop2} and \ref{prop3} there is,

\begin{equation}\label{eq:j4}
\begin{split}
\int_0^t |\mathcal{E}_4(t')| \; dt' \lesssim &  \; \big\||\nabla|^{-2} \mathcal{G} \big\|_{L_t^{\infty}(H^3)} \cdot \int_0^t \big\||\nabla|^{-1} (q \nabla \cdot u)\big\|_{H^3} \; dt'\\
\lesssim & \; \big\|q(\rho)\big\|_{L^\frac{6}{3-2s_0}\cap \dot H^2}  \big(\mathcal{V}(t)^{\frac{1}{2}} + \mathcal{W}(t)^{\frac{1}{2}} \big) \mathcal{E}(t)^\frac{1}{2}\\
\lesssim & \; \big\|\nabla^{s_0} \rho \big \|_{\dot H^{2-s_0}}\big(\mathcal{V}(t)^{\frac{1}{2}} + \mathcal{W}(t)^{\frac{1}{2}} \big) \mathcal{E}(t)^\frac{1}{2} \\
\lesssim & \; \big(\mathcal{V}(t)^{\frac{1}{2}} + \mathcal{W}(t)^{\frac{1}{2}} \big) \mathcal{E}(t)^\frac{1}{2} \mathcal{A}(t)^{\frac{1}{2}}.
\end{split}
\end{equation}

Finally, we turn to the last term  $\mathcal{E}_5$. Through simple calculation we can obtain,

\begin{equation}\nonumber
\begin{split}
\mathcal{E}_5 = & -\sum_{k = 0}^{2} \int_{\mathbb{R}^3} \nabla^k |\nabla|^{-2}  \nabla \cdot (u \cdot \nabla G) \nabla^k |\nabla|^{-2} \mathcal{G} \; dx
- \int_{\mathbb{R}^3} \nabla \nabla \cdot (u \cdot \nabla G) \nabla  \mathcal{G} \; dx \\
 = & -\sum_{k = 0}^{2} \int_{\mathbb{R}^3} \nabla^k |\nabla|^{-2} \nabla \cdot \Big[ \nabla \cdot (u G) - \nabla \cdot u G \Big] \nabla^k |\nabla|^{-2} \mathcal{G}\; dx \\
&\quad -\int_{\mathbb{R}^3} \nabla (\nabla u \cdot \nabla G + u\cdot \nabla \mathcal{G} )\nabla  \mathcal{G} \; dx .
\end{split}
\end{equation}
Hence, using  H\"{o}lder inequality and Sobolev imbedding theorem,

\begin{equation}\nonumber
\begin{split}
|\mathcal{E}_5| \lesssim & \; \Big(\|uG\|_{H^2} + \big\| |\nabla|^{-1} \big[{\nabla \cdot u G} \big] \big\|_{H^2}\Big) \big\||\nabla|^{-2} \mathcal{G}\big\|_{H^2} \\
& + \big\|\nabla(\nabla u\cdot \nabla G) \big\|_{L^2} \|\nabla \mathcal{G}\|_{L^2} + \|\nabla u\|_{L^\infty} \|\nabla \mathcal{G}\|_{L^2}^2.
\end{split}
\end{equation}
Noting here holds that,
\begin{equation}\nonumber
\begin{split}
\|u G\|_{H^2} \lesssim  & \; \|u\|_{L^\frac{3}{s_0}}\|G\|_{L^\frac{6}{3-2s_0}} + \|u\|_{L^\infty} \|\nabla^2 G\|_{L^2} \\
 & \; + \|G\|_{L^\infty} \|\nabla^2 u\|_{L^2}+ \| \nabla G\|_{L^3}\| \nabla u\|_{L^6}, \\
{ \||\nabla|^{-1}[ \nabla \cdot u G]\|_{H^2} \lesssim} & {\|\nabla \cdot u\|_{L^\frac{3}{1+s_0}} \|G\|_{L^\frac{6}{3-2s_0}} + (\|G\|_{W^{1,3}} + \|G\|_{L^\infty}) \|\nabla^2 u\|_{L^2},}
\end{split}
\end{equation}
and,
\begin{equation}\nonumber
\|\nabla (\nabla\cdot u \nabla G)\|_{L^2} \lesssim  \;  \|\nabla u\|_{H^2} \|\nabla G \|_{H^1}.
\end{equation}
Recall that  constant $0 < s_0 < 1/2$,  using {Propositions \ref{prop1}, \ref{prop2} and \ref{prop3}} together, we can derive the estimate as follows,
\begin{equation}\label{eq:j5}
\begin{split}
& \int_0^t |\mathcal{E}_5(t')| \; dt' \\
\lesssim  & \; \big( \mathcal{V}(t)^{\frac{1}{2}} + \mathcal{W}(t)^{\frac{1}{2}} \big)\mathcal{E}(t)^{\frac{1}{2}} \|G \|_{L^\infty_t(W^{1,3} \cap L^\infty)}
+ \mathcal{E}(t)  \mathcal{W}(t)^{\frac{1}{2}}  \\
\lesssim & \; \big( \mathcal{V}(t)^{\frac{1}{2}} + \mathcal{W}(t)^{\frac{1}{2}} \big)\mathcal{E}(t)^{\frac{1}{2}} {(\|F -I \|_{L^\infty_t(\dot H^{s_0} \cap \dot H^2)} + \|\rho\|_{L^\infty_t(\dot H^{s_0} \cap \dot H^2)}) }
+ \mathcal{E}(t)  \mathcal{W}(t)^{\frac{1}{2}}  \\
\lesssim & \; \big( \mathcal{V}(t)^{\frac{1}{2}} + \mathcal{W}(t)^{\frac{1}{2}} \big)\mathcal{E}(t)^{\frac{1}{2}} \mathcal{A}(t)^{\frac{1}{2}}
+ \mathcal{E}(t)  \mathcal{W}(t)^{\frac{1}{2}} .
\end{split}
\end{equation}

Now, we combine the estimates for $\mathcal{E}_1 \thicksim \mathcal{E}_5$, namely \eqref{eq:j1}, \eqref{eq:j2}, \eqref{eq:j3}, \eqref{eq:j4} and \eqref{eq:j5}.
Integrating equality \eqref{E0:eq2} with time, we then derive the bound of $\tilde{\mathcal{E}}(t)$ defined in \eqref{e01}.

\begin{equation}\nonumber
\tilde{\mathcal{E}}(t)  \lesssim \;  \tilde{\mathcal{E}}(0) + \mathcal{E}_{total}(t)^{\frac{3}{2}}.
\end{equation}

\textbf{Step Two:}

In the next step, we focus on the time integral of \textit{effective flux} $\mathcal{G}$ and derive the \textit{a priori} estimate for it.
Operating derivatives $\nabla^k |\nabla|^{-1} (k = 0, 1, 2)$ on the equation of velocity \eqref{u}, then taking inner product with $\nabla^k |\nabla|^{-1} \mathcal{G}$, summing them up it yields that,

\begin{equation}\nonumber
\begin{split}
\big\||\nabla|^{-1} \mathcal{G} \big\|_{H^2}^2 =\mathcal{E}_6 + \mathcal{E}_7 + \mathcal{E}_8,
\end{split}
\end{equation}
in which,
\begin{equation}\nonumber
\begin{split}
\mathcal{E}_6 = & {-}\sum_{k = 0}^2 \int_{\mathbb{R}^3} \nabla^k |\nabla|^{-1} \Delta u \nabla^k |\nabla|^{-1} \mathcal{G} \; dx,\\
\mathcal{E}_7 = &\sum_{k = 0}^{2}\int_{\mathbb{R}^3} \nabla^k |\nabla|^{-1} \big [ u\cdot \nabla u {-} f(\rho) (\mathcal{G}+ \Delta u)\big] \nabla^k |\nabla|^{-1} \mathcal{G} \; dx, \\
\mathcal{E}_8 = & \sum_{k = 0}^{2}\int_{\mathbb{R}^3} \nabla^k |\nabla|^{-1}u_t \nabla^k |\nabla|^{-1} \mathcal{G} \; dx.
\end{split}
\end{equation}

To give the estimate of $\mathcal{E}_6$, we just apply H\"older inequality and derive,

\begin{equation}\label{eq:j6}
\begin{split}
\int_0^t |\mathcal{E}_6(t')| \; dt' & \lesssim \; \int_0^t \|\nabla u\|_{H^2} \big\||\nabla|^{-1} \mathcal{G} \big\|_{H^2} \; dt'  \\
& \lesssim \; \tilde{\mathcal{E}}(t)^{\frac{1}{2}} \Big( \int_0^t  \big\||\nabla|^{-1} \mathcal{G} \big\|_{H^2}^2 \; dt' \Big)^{\frac{1}{2}} .
\end{split}
\end{equation}

For the second term $\mathcal{E}_7$, applying Proposition \ref{prop3} and Propostion \ref{prop1}, it's natural to get,
\begin{equation} \nonumber
\begin{split}
|\mathcal{E}_7| \lesssim & \; \Big( \big\||\nabla|^{-1} (u\cdot \nabla u)  \big\|_{H^2} + \big\||\nabla|^{-1}[f(\rho)(\mathcal{G}+\Delta u)] \big\|_{H^2}\Big) \big\||\nabla|^{-1} \mathcal{G} \big\|_{H^2} \\
\lesssim & \; \big(\|\nabla u\|_{\dot H^{\frac{1}{2} - s_0}} \|u\|_{\dot H^{s_0}} + \|u\|_{L^\infty \cap \dot H^1} \|\nabla u\|_{L^\infty \cap \dot H^1} \big) \big\||\nabla|^{-1} \mathcal{G} \big\|_{H^2} \\
& +\big\{ \|\mathcal{G} + \Delta u\|_{\dot H^{\frac{1}{2} - s_0}} \|\rho\|_{\dot H^{s_0}} + \|\nabla(\mathcal{G} + {\Delta u})\|_{L^2} \|\rho\|_{L^\infty} \\
 & \qquad + \|\nabla \rho \|_{L^3} \|(\mathcal{G} + {\Delta u})\|_{L^6} \big\} \big\||\nabla|^{-1} \mathcal{G} \big\|_{H^2}.
\end{split}
\end{equation}
According to Proposition \ref{prop2} and Sobolev imbeding theorem,

\begin{equation}\label{eq:j7}
\begin{split}
\int_0^t |\mathcal{E}_7(t')| \; dt' \lesssim \tilde{\mathcal{E}}(t) \big\||\nabla|^{-1} \mathcal{G} \big\|_{L^2_t(H^2)} +  \tilde{\mathcal{E}}(t)^{\frac{1}{2}}  \mathcal{A}(t)^{\frac{1}{2}} \big\||\nabla|^{-1} \mathcal{G} \big\|_{L^2_t(H^2)}.
\end{split}
\end{equation}

For the last term $\mathcal{E}_8$, we can use integration by parts and split it into the following two terms,

\begin{equation}\nonumber
\begin{split}
\mathcal{E}_8 & = \frac{d}{dt} \sum_{k = 0}^2 \int_{\mathbb{R}^3} \nabla^k |\nabla|^{-1} u \nabla^k |\nabla|^{-1} \mathcal{G} \; dx
-  \sum_{k = 0}^2 \int_{\mathbb{R}^3} \nabla^k |\nabla|^{-1} u \nabla^k |\nabla|^{-1} \mathcal{G}_t \; dx \\
& \triangleq \mathcal{E}_{8,1} + \mathcal{E}_{8,2}.
\end{split}
\end{equation}
Noting the second equation of system \eqref{E0:eq1}, according to Proposition \ref{prop3} and Proposition \ref{prop1} there is,

\begin{equation}\nonumber
\begin{split}
& |\mathcal{E}_{8,2}| \\
\lesssim & \; \Big |\sum_{k = 0}^2 \int_{\mathbb{R}^3} \nabla^k |\nabla|^{-1} u \nabla^k |\nabla|^{-1}\big[\Delta \Gamma^2 u - \nabla \cdot (u\cdot \nabla G) \\
& \qquad + \nabla (q \nabla \cdot u) - \nabla \cdot g(\nabla u, F)\big] \; dx \Big| \\
\lesssim & \; \|u\|_{H^2}^2 + \|u\|_{H^2} \Big\{ \|u G\|_{H^2} + \big\||\nabla|^{-1} ({\nabla \cdot u G}) \big\|_{H^2} \\
& \qquad +  \big\||\nabla|^{-1} (\nabla \cdot u q) \big\|_{H^2} +  \big\||\nabla|^{-1} g (\nabla u, F) \big\|_{H^2}\Big\} \\
\lesssim & \; \|u\|_{H^2}^2 + \|u\|_{H^2} \Big\{ \|\nabla u\|_{H^1} \| G\|_{L^3 \cap \dot H^2} + \|\nabla u\|_{\dot H^{\frac{1}{2} - s_0}} \big(\|F-I\|_{\dot H^{s_0}} + \|\rho\|_{\dot H^{s_0}}\big) \\
& \qquad + \|\nabla u\|_{L^\infty \cap \dot H^1} \big( \|\rho\|_{L^\infty \cap \dot H^1} + \|F-I\|_{L^\infty \cap \dot H^1} \big)\Big\}.
\end{split}
\end{equation}
Hence, integrating $|\mathcal{E}_8|$ over $(0,t)$ and using Sobolev imbedding theorem and Proposition \ref{prop2}, we shall get,
\begin{equation}\label{eq:j8}
\begin{split}
{\Big | \int_0^t  \mathcal{E}_8(t') \; dt' \Big|} \lesssim & \; {\Big | \int_0^t \mathcal{E}_{8,1}(t') \; dt' \Big|}+ \int_0^t |\mathcal{E}_{8,2}(t')| \; dt' \\
\lesssim & \; \sup_{0 \leq t' \leq t} \big\||\nabla|^{-1} u \big\|_{H^2} \big\||\nabla|^{-1} \mathcal{G} \big\|_{H^2} + \tilde{\mathcal{E}}(t) + \mathcal{V}(t)\mathcal{A}(t)^{\frac{1}{2}} \\
\lesssim & \; \sup_{0 \leq t' \leq t}\big\||\nabla|^{-1} \mathcal{G} \big\|_{H^2}\tilde{\mathcal{E}}(t)^{\frac{1}{2}} + \tilde{\mathcal{E}}(t) + \mathcal{V}(t)\mathcal{A}(t)^{\frac{1}{2}}.
\end{split}
\end{equation}
Combing the estimates \eqref{eq:j6}, \eqref{eq:j7} and \eqref{eq:j8} together, we shall get the bound of $\int_0^t \| |\nabla|^{-1}\mathcal{G}\|_{H^2}^2 dt'$.
Using Young's inequality and Proposition \ref{prop0}, we then complete the proof of this lemma.

\end{proof}

\subsection{The estimate of strongly dissipative energy $\mathcal{W}(t)$}

In this subsection, we will move on to handle the strongly dissipative energy  $\mathcal{W}(t)$. We give the following lemma.

\begin{lemma}\label{lemE2}
Assume that energies are defined as in Section 1, we then have the following inequality,

\begin{equation}\nonumber
\mathcal{W}(t) \lesssim \; \mathcal{W}(0) + \mathcal{E}(t)  + \mathcal{E}_{total}(t)^{\frac{3}{2}},
\end{equation}
holds for any positive time $t$.
\end{lemma}

\begin{proof}

At the beginning of proof, we recall the system of $(\mathbf{\Gamma} u, \mathcal{G})$ in which $f(\rho)=-\frac{\rho}{\rho+1}$.

    \begin{equation}\label{E2:eq1}
    \left\{
      \begin{array}{ll}
        (\mathbf{\Gamma} u)_t +\mathbf{ \Gamma} (u\cdot\nabla u) -\Delta \mathbf{\Gamma} u  - \mathbf{\Gamma} \big\{ f(\rho)(\mathcal{G}+\Delta u) \big\} = \mathbf{\Gamma}  \mathcal{G} ,  \\
    \mathcal{G}_t + \nabla \cdot (u \cdot \nabla G)- \nabla (q \nabla \cdot u) + \nabla \cdot g(\nabla u, F) = \Delta \mathbf{\Gamma}^2 u.
      \end{array}
    \right.
    \end{equation}
Like the operation in the proof of Lemma \ref{lemE0}, we still separate the process into two steps. Firstly, we define partial strongly dissipative energy $\tilde{\mathcal{W}}(t)$ in the following,

\begin{equation}\label{E2:eq2}
\tilde{\mathcal{W}}(t) \triangleq   \sup_{0 \leq t' \leq t} \Big\{(1+t')^2 \big( \|\nabla \mathbf{\Gamma} u\|_{H^1}^2 + \| \mathcal{G}\|_{H^1}^2 \big) \Big\}
+ \int_0^t (1+t')^2\|\nabla^2 \mathbf{\Gamma} u\|_{H^1}^2 \; dt' .
\end{equation}
In the next, we shall derive the estimate for $\tilde{\mathcal{W}}(t)$ via energy method.

\textbf{Step One:}

Applying derivatives $\nabla^k$ ($k = 1, 2$) on the first equation of system \eqref{E2:eq1} and taking inner product with $\nabla^k \mathbf{\Gamma} u$.
Applying derivatives $\nabla^{k-1}$ on the second equation of \eqref{E2:eq1} and taking inner product with $\nabla^{k-1} \mathcal{G}$.  Adding the time weight $(1+t')^2$ respectively and summing them up, there is,

\begin{equation}\label{E2:eq3}
\frac{1}{2}\frac{d}{dt} \Big\{ (1+t')^2 \big(\| \nabla \mathbf{\Gamma} u\|_{H^1}^2 + \| \mathcal{G}\|_{H^1}^2 \big)\Big\} + (1+t')^2\|\nabla^2 \mathbf{\Gamma} u\|_{H^1}^2 =  \sum_{i=1}^6 \mathcal{W}_i,
\end{equation}
with,
\begin{equation}\nonumber
\begin{split}
\mathcal{W}_1 = & \; (1+t') \big(\| \nabla \mathbf{\Gamma} u\|_{H^1}^2 + \| \mathcal{G}\|_{H^1}^2 \big),\\
\mathcal{W}_2 = & \; (1+t')^2\sum_{k = 1}^{2} \Big\{ \int_{\mathbb{R}^3} \nabla^k \mathbf{\Gamma} \mathcal{G}\nabla^k  \mathbf{\Gamma} u \; dx + \int_{\mathbb{R}^3} \nabla^{k-1} \Delta \mathbf{\Gamma}^2 u  \nabla^{k-1} \mathcal{G} \; dx\Big\},\\
\mathcal{W}_3 = & -(1+t')^2\sum_{k = 1}^{2}\Big\{\int_{\mathbb{R}^3} \nabla^k \mathbf{\Gamma} (u \cdot \nabla u) \nabla^k  \mathbf{\Gamma} u\; dx \\
& \qquad +
\int_{\mathbb{R}^3} \nabla^{k-1}  \nabla \cdot g(\nabla u, F) \nabla^{k-1}  \mathcal{G}\; dx \Big\}, \\
\mathcal{W}_4 =  & \; (1+t')^2\sum_{k = 1}^{2} \int_{\mathbb{R}^3} \nabla^k  \mathbf{\Gamma} \Big\{ f(\rho)(\mathcal{G}+\Delta u ) \Big\} \nabla^k  \mathbf{\Gamma} u\; dx,\\
\mathcal{W}_5 = & \; (1+t')^2\sum_{k = 1}^{2} \int_{\mathbb{R}^3} \nabla^{k-1} \nabla (q \nabla \cdot u) \nabla^{k-1} \mathcal{G} \; dx,\\
\mathcal{W}_6 = & \; -(1+t')^2\sum_{k = 1}^{2} \int_{\mathbb{R}^3} \nabla^{k-1} \nabla \cdot (u \cdot \nabla G) \nabla^{k-1} \mathcal{G} \; dx.
\end{split}
\end{equation}

For the first term $\mathcal{W}_1$, through Proposition \ref{prop0} we directly write,

\begin{equation}\label{eq:k1}
\int_0^t |\mathcal{W}_1(t')| \; dt' \lesssim \; \mathcal{V}(t) \lesssim \; \mathcal{E}(t)^{\frac{1}{2}}\mathcal{W}(t)^{\frac{1}{2}}.
\end{equation}

The estimate for $\mathcal{W}_2$ is similar to $\mathcal{E}_1$. We use some cancellation to extinguish this term. Through integration by parts, there is,
\begin{equation}\label{eq:k2}
\begin{split}
\mathcal{W}_2 = & \;(1+t')^2\sum_{k = 1}^{2}\Big\{ \int_{\mathbb{R}^3}  \nabla^k \mathbf{\Gamma} \mathcal{G} \nabla^k  \mathbf{\Gamma }u \; dx + \int_{\mathbb{R}^3}  \nabla^{k-1} \Delta \mathbf{\Gamma} u  \nabla^{k-1}\mathbf{\Gamma} \mathcal{G} \; dx\Big\},\\
= & \;(1+t')^2\sum_{k = 1}^{2}\Big\{ \int_{\mathbb{R}^3}  \nabla^k \mathbf{\Gamma} \mathcal{G} \nabla^k  \mathbf{\Gamma} u \; dx - \int_{\mathbb{R}^3}  \nabla^k  \mathbf{\Gamma} u  \nabla^k\mathbf{\Gamma} \mathcal{G} \; dx\Big\}\\
= & \; 0.
\end{split}
\end{equation}

Also, using integration by parts and H\"{o}lder inequality, we can get the estimate for $\mathcal{W}_3$,{
\begin{equation}\nonumber
\begin{split}
 |\mathcal{W}_3| \lesssim & \;  (1+t')^2  \Big(\|u\cdot \nabla u\|_{H^1} \|\nabla^2 u\|_{H^1} + (\|g(\nabla u, F)\|_{L^2} + \|\nabla \nabla \cdot g(\nabla u , F)\|_{L^2}) \|\nabla \mathcal{G}\|_{L^2} \Big)\\
\lesssim & \;
(1+t')^2 \Big(\|u\|_{L^\infty}\|\nabla^2 u\|_{L^2} + \|\nabla u\|_{L^2} \|\nabla u\|_{L^\infty} + \| u\|_{L^3} \|\nabla u\|_{L^6} \Big)\|\nabla^2 u\|_{H^1} \\
&+
(1+t')^2 \Big(\|F-I\|_{L^3}\|\nabla u\|_{L^6} + \|F-I\|_{L^\infty}\|\nabla^3 u\|_{L^2} + \|\nabla^2 u\|_{L^6} \|\nabla F\|_{L^3} \\
& + \|\nabla u\|_{L^\infty}\|\nabla^2(F-I)\|_{L^2} \Big)\|\nabla \mathcal{G}\|_{L^2}\\
\lesssim & \; (1+t')^2 \|\nabla^2 u\|_{H^1}^2 \|u\|_{H^2} + (1+t')^2 \|\nabla^2 u\|_{H^1} \|\nabla \mathcal{G}\|_{L^2} \|\nabla^{s_0} (F-I)\|_{H^{2-s_0}}.
\end{split}
\end{equation}
}
Which implies,
\begin{equation}\label{eq:k3}
\begin{split}
&\int_0^t |\mathcal{W}_3(t')| \;dt' \lesssim \; \mathcal{W}(t) \big( \mathcal{E}(t)^{\frac{1}{2}} + \mathcal{A}(t)^{\frac{1}{2}}\big).
\end{split}
\end{equation}

For the next term $\mathcal{W}_4$, we apply H\"{o}lder inequality and Proposition \ref{prop1}, it then becomes,

\begin{equation}\nonumber
\begin{split}
 |\mathcal{W}_4|\lesssim & \; (1+t')^2 \big\|f(\rho)(\mathcal{G}+\Delta u) \big\|_{H^1} \|\nabla^2 u\|_{H^1},\\
 \lesssim & \; (1+t')^2 \|f(\rho)\|_{W^{1,3} \cap L^\infty } \|\mathcal{G} + \Delta u \|_{\dot H^1} \|\nabla^2 u\|_{H^1}\\
 \lesssim & \; (1+t')^2\|\nabla^{s_0} \rho\|_{H^{2 - s_0}} \|\mathcal{G} + \Delta u \|_{\dot H^1} \|\nabla^2 u\|_{H^1}.
\end{split}
\end{equation}
Obviously it yields,

\begin{equation}\label{eq:k4}
\begin{split}
& \int_0^t |\mathcal{W}_4(t')|\; dt' \\
\lesssim & \; \sup_{0 \leq t' \leq t} \big\|\nabla^{s_0} \rho \big\|_{H^{2 - s_0}} \cdot \int_0^t (1+t')^2 \big\|\mathcal{G} + \Delta u \big\|_{\dot H^1} \|\nabla^2 u\|_{H^1}\; dt'\\
\lesssim & \; \mathcal{E}_{total}(t)^{\frac{3}{2}}.
\end{split}
\end{equation}

Similarly, using Sobolev imbedding theorem we can give the estimate for $\mathcal{W}_5$,

\begin{equation}\label{eq:k5}
\begin{split}
& \int_0^t |\mathcal{W}_5(t')| \; dt' \\
 \lesssim & \; \int_0^t (1+t')^2 \Big(\|q(\rho)\|_{L^3}\|\nabla \cdot u\|_{L^6} + \|\nabla q\|_{\dot H^1}
\|\nabla^2 u\|_{\dot H^1}\Big) \|\nabla \mathcal{G}\|_{L^2} \;dt'\\
\lesssim & \; \mathcal{E}_{total}(t)^{\frac{3}{2}}.
\end{split}
\end{equation}

The last term is $\mathcal{W}_6$ and we can write,

$$ \mathcal{W}_6 = \; -(1+t')^2 \sum_{k = 1}^2\int_{\mathbb{R}^3} \nabla^{k - 1}\big(u\cdot \nabla \mathcal{G} + \nabla u \cdot \nabla G\big) \nabla^{k - 1}\mathcal{G} \; dx. $$
Using integration by parts, Sobolev imbedding theorem and Proposition \ref{prop1} again,

\begin{equation}\nonumber
\begin{split}
|\mathcal{W}_6| \lesssim & \; (1+t')^2 \big({\|\nabla \cdot u\|_{L^\infty} \|\mathcal{G}\|_{L^2}^2} + \|\nabla^2 u\|_{L^2} \|G\|_{L^3} \|\mathcal{G}\|_{L^6} + \|\nabla u\|_{L^6} \|G\|_{L^3} \|\nabla \mathcal{G}\|_{L^2}\big)\\
& + (1+t')^2 \Big\{\|\nabla  u\|_{L^\infty} \|\nabla \mathcal{G}\|_{L^2}^2 + \|\nabla^2 u\|_{L^6} \|\nabla G\|_{L^3}\|\nabla \mathcal{G}\|_{L^2} \\
& \quad + \|\nabla u\|_{L^\infty} \|\nabla^2 G\|_{L^2}\|\nabla \mathcal{G}\|_{L^2}\Big\}\\
\lesssim & \; {(1+t')^2 \|\mathcal{G}\|_{L^2}^2 \|\nabla^2  u\|_{H^1} }\\
& { + (1+t')^2 \|\nabla^2 u\|_{H^1} (\||\nabla|^{s_0} (F-I)\|_{\dot H^{2-s_0}} + \||\nabla|^{s_0} \rho \|_{\dot H^{2-s_0}})\|\nabla \mathcal{G}\|_{L^2}}.
\end{split}
\end{equation}
Hence, applying Proposition {\ref{prop2}} we will get,
\begin{equation}\label{eq:k6}
\begin{split}
& \int_0^t |\mathcal{W}_6 (t')| \; dt' \\
\lesssim & \; {\sup_{0 \leq t' \leq t} (1+t')^2 \|\mathcal{G}\|_{L^2}^2  \cdot \int_0^t \|\nabla^2  u\|_{H^1} \; dt'}\\
 & +  \sup_{0 \leq t' \leq t} \|\nabla^{s_0} G\|_{H^{2-s_0}}  \cdot \int_0^t (1+t')^2 \|\nabla^2 u\|_{\dot H^1} \|\nabla \mathcal{G}\|_{L^2}\; dt' \\
\lesssim & \; \mathcal{E}_{total}(t)^{\frac{3}{2}}.
\end{split}
\end{equation}

Now, we combine the estimates for $\mathcal{W}_1 \thicksim \mathcal{W}_6$, namely \eqref{eq:k1}, \eqref{eq:k2}, \eqref{eq:k3}, \eqref{eq:k4}, \eqref{eq:k5} and \eqref{eq:k6}.
Integrating equality \eqref{E2:eq3} with time, we then derive the bound of $\tilde{\mathcal{W}}(t)$ defined in \eqref{E2:eq2}.

\begin{equation}\nonumber
\tilde{\mathcal{W}}(t)  \lesssim \;  \tilde{\mathcal{W}}(0) + \mathcal{E}(t)^\frac{1}{2}\mathcal{W}(t)^\frac{1}{2} + \mathcal{E}_{total}(t)^{\frac{3}{2}}.
\end{equation}

\textbf{Step Two:}

Next, we focus on the time integral of $\nabla \mathcal{G}$ and derive the \textit{a priori} estimate for it.
Operating derivative $\nabla $ on \eqref{u} and taking inner product with $\nabla \mathcal{G}$.
Adding the time weight $(1+t')^2$ and summing them up,  it then becomes,

\begin{equation}\nonumber
\begin{split}
  (1+t')^2 \|\nabla \mathcal{G}\|_{L^2}^2 = \mathcal{W}_7 + \mathcal{W}_8+\mathcal{W}_9.
\end{split}
\end{equation}
In which,

\begin{equation}\nonumber
\begin{split}
\mathcal{W}_7= & \; (1+t')^2\int_{\mathbb{R}^3} \nabla \Delta u \nabla \mathcal{G}\; dx,\\
 \mathcal{W}_8 = & \; (1+t')^2\int_{\mathbb{R}^3} \nabla \Big [ u\cdot \nabla u + f(\rho) \big(\mathcal{G}+ \Delta u \big)\Big] \nabla \mathcal{G} \; dx,  \\
\mathcal{W}_9 = &  \; (1+t')^2\int_{\mathbb{R}^3} \nabla u_t \nabla \mathcal{G} \; dx.
\end{split}
\end{equation}

For the term $\mathcal{W}_7$, it directly shows that,

\begin{equation}\label{eq:k7}
\begin{split}
\int_0^t |\mathcal{W}_7(t')| \; dt' \lesssim & \; \int_0^t (1+t')^2\|\nabla^2 u\|_{L^2} \|\nabla \mathcal{G}\|_{L^2} \; dt' \\
\lesssim & \; \tilde{\mathcal{W}}(t)^{\frac{1}{2}} \Big( \int_0^t  (1+t')^2\|\nabla \mathcal{G}\|_{L^2}^2 \; dt' \Big)^{\frac{1}{2}}.
\end{split}
\end{equation}

Applying {H\"{o}lder's inequality and Sobolev imbedding theorem}, we can derive,
\begin{equation} \label{eq:k8}
\begin{split}
& \int_0^t |\mathcal{W}_8(t')|\;dt' \\
\lesssim & \; \int_0^t (1+t')^2 \|u\|_{L^\infty \cap \dot H^1} \|\nabla u\|_{L^\infty \cap \dot H^1} \|\nabla \mathcal{G}\|_{L^2} \; dt' \\
 & \quad  +    \int_0^t (1+t')^2 {\big( \|f(\rho)\|_{L^\infty} + \|\nabla f(\rho)\|_{L^3})} \|\mathcal{G}+\Delta u\|_{\dot H^1} \|\nabla \mathcal{G}\|_{L^2} \; dt' \\
\lesssim & \; \mathcal{E}_{total}(t)^{\frac{3}{2}}.
\end{split}
\end{equation}

For the last term $\mathcal{W}_9$, we still use integration by parts and split it into three terms,

\begin{equation}\nonumber
\begin{split}
\mathcal{W}_9  = &\; \frac{d}{dt} \Big\{ (1+t')^2 \int_{\mathbb{R}^3} \nabla u \nabla \mathcal{G} \; dx \Big\}
-  2(1+t') \int_{\mathbb{R}^3} \nabla u \nabla \mathcal{G} \; dx - \int_{\mathbb{R}^3} \nabla u \nabla \mathcal{G}_t \; dx \\
\triangleq & \; \mathcal{W}_{9,1} + \mathcal{W}_{9,2} + \mathcal{W}_{9,3}.
\end{split}
\end{equation}
Noting the second equation of system \eqref{E2:eq1} we then have,

\begin{equation}\nonumber
\begin{split}
 |\mathcal{W}_{9,3}|
\lesssim & \; (1+t')^2 \Big | \int_{\mathbb{R}^3} \nabla  u \nabla \Big\{ \Delta \Gamma^2 u - \nabla \cdot (u\cdot \nabla G) \\
&\qquad\qquad  + \nabla (q \nabla \cdot u) - \nabla \cdot g(\nabla u, F) \Big\} \; dx \Big| \\
\lesssim & \; (1+t')^2 \|\nabla^2 u\|_{L^2} \Big(\|\nabla^2 u\|_{L^2} + \|u\|_{L^\infty} \|\nabla \mathcal{G} \|_{L^2} + \|\nabla u\|_{L^6}
\|\nabla G\|_{L^2} \\
&+ \|\nabla q(\rho)\|_{L^3} \|\nabla \cdot u\|_{L^6} + \|q(\rho)\|_{L^\infty}\|\nabla \nabla \cdot u\|_{L^2}\\
& { + \|\nabla^2 u\|_{L^2} \|F-I\|_{L^\infty} + \|\nabla u\|_{L^6} \|\nabla F\|_{L^3} } \Big)\\
\lesssim& \; (1+t')^2   \Big\{  \|\nabla^2 u\|_{L^2}^2 + \|\nabla^2 u\|_{L^2}\|\nabla \mathcal{G}\|_{L^2}\|u\|_{H^2}\\
 &{+ \|\nabla^2 u\|_{L^2}^2 \big(\|\nabla^{s_0} \rho\|_{H^{2-s_0}} + \|\nabla^{s_0} (F-I)\|_{H^{2-s_0}}\big)} \Big\}.
\end{split}
\end{equation}
Indeed, there is,

\begin{equation}\label{eq:k9}
\begin{split}
{\Big | \int_0^t \mathcal{W}_9(t')  \; dt' \Big| }\lesssim & \; {| \int_0^t  \mathcal{W}_{9,1}(t')  \; dt'|} + \int_0^t | \mathcal{W}_{9,2}(t')|  \; dt' +\int_0^t | \mathcal{W}_{9,3}(t')|  \; dt'  \\
\lesssim & \; \tilde{\mathcal{W}}(t) + \mathcal{V}(t) +\mathcal{E}_{total}(t)^{\frac{3}{2}} .
\end{split}
\end{equation}

Combing the estimates \eqref{eq:k7}, \eqref{eq:k8} and \eqref{eq:k9} together, we then get the bound of $\int_0^t (1+t')^2\|\nabla\mathcal{G}\|_{L^2}^2 dt'$.
Using Young's inequality and  Proposition \ref{prop0}, we then complete the proof of this lemma.

\end{proof}

\section{Proof of Theorem \ref{thm1}}

In this section, we will combine the above \textit{a priori} estimates for the basic energy $\mathcal{E}(t)$, the strongly dissipative energy $\mathcal{W}(t)$ and the assistant energy $\mathcal{A}(t)$ together. Finally we give the proof of Theorem \ref{thm1}.

Recall that the total energy defined in Section 3 can be written as,
$$ \mathcal{E}_{total}(t) \triangleq \mathcal{E}(t)+\mathcal{V}(t)+\mathcal{W}(t)+\mathcal{A}(t).$$
Due to Lemma \ref{lem},  Lemma \ref{lemE0} and  Lemma \ref{lemE2},  noting the Proposition \ref{prop0} and Young's inequality, we can find some positive constant $C^*$ and there holds,

\begin{equation}\label{ee}
\mathcal{E}_{total}(t) \leq C^* \mathcal{E}_{total}(0)+ C^* \mathcal{E}_{total}(t)^{\frac{3}{2}}.
\end{equation}

Using the setting of initial data in Theorem \ref{thm1}, indeed the condition \eqref{initial}, if the constant $\epsilon$ is small enough,  there is,
\begin{equation}\nonumber
 C^* \mathcal{E}_{total}(0) \leq \frac{\epsilon}{2}.
\end{equation}
Through the local existence theory which can be achieved via standard energy method, there exists a positive time $T$ and,

\begin{equation}\label{t}
 \mathcal{E}_{total}(t) \leq \epsilon,  \qquad \forall t \in [0,T].
\end{equation}
Let $T^*$ be the largest possible time of $T$ for which \eqref{t} holds. We then need to show the fact $T^* = \infty$. Using standard continuation argument, if $\epsilon$ is small enough, the inequality \eqref{ee} implies the conclusion.  We omit the details here and complete the proof of Theorem \ref{thm1}.

\section*{Acknowledgement}
The author sincerely appreciates the helpful suggestion from Professor Zhen Lei.
 The  author is supported by Shanghai Sailing Program (18YF1405500), Fundamental Research Funds for the Central Universities (222201814026), China Postdoctoral Science Foundation (2018M630406, 2019T120308) and NSFC (11801175).

\end{document}